\documentclass[12pt,reqno]{amsart}

\usepackage{amsmath,amsthm,amssymb, mathrsfs, bbm}

\theoremstyle{plain}
\newtheorem{theorem}{Theorem}[section]
\newtheorem{corollary}[theorem]{Corollary}

\newtheorem{proposition}[theorem]{Proposition}

\theoremstyle{definition}

\numberwithin{equation}{section}

\newcommand{\N}{\mathbb{N}}

\newcommand{\R}{\mathbb{R}}
\newcommand{\C}{\mathbb{C}}

\newcommand{\D}{\mathbb{D}}

\begin{document}
\title{Superresolution of principal semi-algebraic sets}

\author{Mihai Putinar}
\address{University of California at Santa Barbara, CA,
USA and Newcastle University, Newcastle upon Tyne, UK} 

\email{\tt mputinar@math.ucsb.edu, mihai.putinar@ncl.ac.uk}




\begin{abstract} The H\"older continuity of the truncated moment map of a shade function in Euclidean space is established in the vicinity of a principal semi-algebraic set.
The proof combines volume bounds of semi-algebraic sets and convex optimization methods. The main
estimate is applied to a potential type transform specific to two real variables, for perturbations of quadrature domains.

\end{abstract}

\maketitle

\section{Introduction}  Among all basic semi-algebraic subets of $\R^n$, that is sets defined by finitely many
polynomial equalities and inequalities, those defined
by a single inequality $p(x) \geq 0, \ x \in \R^n,$ stand aside. We call them {\it principal semi-algebraic sets}.

While some geometric features of these sets are available via elementary real algebra, their distinctive feature was discovered indirectly,
as extremal solutions to a variational inequality, in itself streaming from a classical moment problem. For instance it is easy to verify that the
positive orthant in two dimensions
$$ E_1 = \{ (x,y) \in \R^2, \ x \geq 0, y \geq 0 \},$$
cannot be defined by a single polynomial inequality, while the union of two opposed orthants
$$ E_2 = \{ (x,y) \in \R^2, \ xy \geq 0\}$$ does. It is less evident, as we will indicate below the mathematical framework and proof, that a major distinction between the two is their identifiability from finitely moment data (when compared with all shapes carrying a degree of shade).

Finite determinateness is a sought-after characteristic in numerous inverse problems. The present note focuses on a particular case: monomial (power) moment data encoding domains endowed with a shade function.
In the neighborhood of a shape determined by finitely many moments, the $L^1$-norm dependence on the perturbation of the moment data is an essential step in any reconstruction algorithm. General results 
in this direction were obtained in conjunction with the maximum entropy recovering technique \cite{Gamboa-1996} and via convex optimization \cite{Lewis-1996}. In the present note we specialize
Lewis's approach \cite{Lewis-1996} to Lebesgue measure in Euclidean space and polynomial functions, by exploiting some known estimates of volumes of semi-algebraic sets. The main result (Theorem \ref{main})
offers precise H\"older continuity estimates of the truncated moment map. The applied world calls this phenomenon superresolution, hence the title.

The contents is the following. We start with a section recalling some classical aspects of the $L$-problem of moments, as forged by Mark Krein and collaborators almost a century ago. 
Extremal solutions to the $L$-problem of moments are exactly those determined by finitely many data. In the next section we
reproduce some recent volume of polynomial sublevel set, stated at a semi-local level. Then we specialize Lewis proof to polynomial functions.

The last section focuses on low dimensions ($n=1$ and $n=2$), with some observations on potential theoretic and function theoretic transforms of the moment data. 
To be more precise, extremal solutions to the $L$-problem of moments on the line are unions of intervals. The exponential of their Cauchy transform carries in a closed form this determinateness feature: it is a rational function at infinity (an observation going back to A. A. Markov).
A similar operation in two real variables (the exponential of a double Cauchy transform) identifies a dense part of all extremal solutions. The corresponding planar shapes are called quadrature domains. This class of domains remained under close scrutiny for several decades, for a variety of theoretical or quite applied motivations \cite{QD-2005}. Our note provides some continuity statements for this specific exponential transform, under small perturbations, in neighborhoods of quadrature domains.

\section{The L-problem of moments} 
Let $\mu$ denote a positive Borel measure on $\R^n$, rapidly decreasing at infinity. In particular every element of the polynomial algebra $\R[x]$ defines an integrable function with respect to $\mu$. Assume in addition that the support of $\mu$ is not real algebraic, that is 
$$(\int | p | d\mu = 0, \ p \in \R[x]) \ \Rightarrow (p=0).$$
Fix a positive integer $N$ and a positive $L$. A quintessential inverse problem is the following: reconstruct, or approximate, a shade function $g \in L^1(\mu), \ -L \leq g \leq L, \ \mu-a.e.,$
from a finite section of its power moments:
$$ s_\alpha(g) = \int x^\alpha g d\mu, \ \ |\alpha| \leq N.$$
Throughout this note we employ the standard multi-index notation $x^\alpha = x_1^{\alpha_1} \ldots x_n^{\alpha_n}, \  x = (x_1,\ldots,x_n) \in \R^n, \ \
\alpha = (\alpha_1,\ldots,\alpha_n) \in \N^n.$ The set of multi-indices is not necessarily restricted by a degree bound, it can be any finite set.

The collection of moments $s(g) = (s_\alpha(g))_{|\alpha| \leq N}$ fills, with varying $g$, a convex set $K$ in a 
finite dimensional euclidean space $V$. Every linear functional $\Phi$ defined on $V$ is given by a polynomial $p \in \R[x]$ of degree less than or equal to $N$. To be more precise, for $p(x) = \sum_{|\alpha|\leq N} p_\alpha x^\alpha,$
$$ \Phi( s(g)) = \sum_\alpha p_\alpha s_\alpha(g) = \int p g d\mu.$$
Whence
$$ \Phi(s(g)) = \int p g d\mu \leq \| g \|_\infty \| p \|_1 \leq L \| p \|_1,$$
where the infinity norm is taken on the support of the measure $\mu$ and $\| \cdot \|_1 = \| \cdot \|_{1,\mu}$. Clearly, if the point $s(g)$ lies in the interior of the convex set $K$, or the above inequality is strict for at least one linear functional, then the original shade function $g$ is not determined by its measurements $s(g)$. On the contrary, if
$$ \int h g d \mu = \| h \|_1, \ \ \deg(h) \leq N,$$
then necessarily $g(x) = L\  {\rm sgn} (h)$. Thus only black and white pictures (by ad-hoc convention $L$ is black and $-L$ is white), delimited by a {\it single} algebraic equation $h(x) =0$
are determined by the power moments of degree up to $N$. And vice-versa. 

Via a minor renormalization, one can start with shade functions $g \in L^1(\mu)$ subject to the bounds $0 \leq g \leq 1$. Then we infer that $g$ is determined by its
power moments $(s_\alpha(g))_{|\alpha| \leq N}$ if and only if $g = \chi_E$ is the characteristic function of the non-negativity set 
$$ E = \{ x \in \R^n, \ h(x) \geq 0 \}$$
associated to a polynomial $h$ of degree at most $N$. 

Returning to our orthant example, we can take $\mu$ to be the Gaussian measure in two variables, or Lebesgue area measure restricted to the unit disk. In short, for any $N \geq 1$, there exists a measurable function $f, \ 0 \leq f \leq 1,$ different than $\chi_{E_1}$, such that
$$ \int_{E_1} x^\alpha d\mu = \int f x^\alpha d\mu, \ \ |\alpha| \leq N,$$
and then there are infinitely many such $f'$s. On the other hand, if for a measurable function $g, 0 \leq g \leq 1$, one has
$$   \int_{E_2} x^\alpha d\mu = \int g x^\alpha d\mu, \ \ |\alpha| \leq 2,$$
then, quite unexpectedly, $g = \chi_{E_2},$ $\mu$-a.e. .

In a series of a dozen articles, written between 1934 and 1940, Akhiezer and Krein have discovered and enhanced the above convexity and $L^1-L^\infty$
duality arguments, developing for this purpose an abstract framework. It was this theoretical setting where several key concepts of modern functional analysis and optimization theory were born. The two books \cite{Ahiezer-Krein-1962,Krein-Nudelman-1977} collect and systemize their ideas. The moment problem studied and generalized by Akhiezer and Krein had its starting point at ``several (little-known) ideas and problems advanced by the late academician A. A. Markov" to use their own words \cite{Ahiezer-Krein-1962} pg. vii. Nowadays the {\it L-problem of moments}, also known as {\it Markov's moment problem} is resurfacing and finds new applications, see for instance \cite{Diaconis-Freedman-1}. 

From a more constructive point of view, once the uniqueness is established, it is imperative to devise robust reconstruction algorithms from moments. Numerical observations, for instance obtained via parallel or ray tomographic projections, or various field measurements, are equivalent to moment data. In this respect, the relation between the power moments and a principal semi-algebraic set is essential to be understood. Several reconstruction techniques of semi-algebraic sets from moments are known. To mention only a few: maximum entropy \cite{Gamboa-1996}, determination of the algebraic boundary \cite{LP-2015,APST-2019}, finding the vertices of a polygon by matrix analysis \cite{Cuyt-2005,Golub-1999}, multi-variate rational approximation \cite{GHMP-2000}. In general, power moments of a given density are very unstable, unless the support is restricted to a special manifold such as a torus or a sphere. Selecting a convenient basis of polynomials or better adapted systems of functions is a natural step to consider, as for instance in \cite{Fatemi-2016}.

Variants of the $L$-problem of moments on abstract normed spaces, or on measure spaces with some prescribed sets of test functions (in place of monomials)
were thoroughly studied, starting with Krein himself and collaborators, see \cite{Krein-Nudelman-1977} and the references cited there. The relatively recent $L^1-L^\infty$ duality approach via optimization techniques \cite{Lewis-1996} will be relevant for the present note.

\section{Volume of polynomial sublevel sets} We recall some well known estimates of the volume of the sublevel set a polynomial function of several real variables. This is a central topics
of interest in singularity theory, asymptotic analysis of oscillatory integrals and algebraic geometry. The fundamental article by Varchenko \cite{Varchenko-1976} put the geometry of the Newton polyhedron associated to a polynomial at the heart of these estimates. Notable refinements were worked out afterwards, see for instance  \cite{Phong-Stein-2001,Greenblatt-2010}. For smooth level sets the co-area formula yields sharp growth of volume bounds, while for singular boundaries Lojasiewicz inequality is an essential ingredient. We confine ourselves to reproduce from a recent study \cite{Dieu-2018} a precise asymptotic upper-bound of the volume of a sublevel set of a polynomial.

Denote $\Delta = [-1,1]^n$, the unit cube in $\R^n$, and $\Delta_r = [-r,r]^n$ for $r>0$.
Let $p(x) = \sum_{|\alpha| \leq d} p_\alpha x^\alpha$ be a polynomial in $n$ variables. For a positive parameter $\delta$ we denote
$$ V_\delta(p) = \{ x \in \R^n, \ |p(x)|<\delta\}.$$
We are interested in the asymptotic upper bounds of the Lebesgue measure ${\rm vol}(V_\delta(p) \cap \Delta_r).$ To this aim we call a multi-index $\alpha \in \N^n$
{\it admissible with respect to} $p$, if $p_\alpha \neq 0$ and there exists a permutation $(\sigma(1), \sigma(2), \ldots, \sigma(n))$ of $(1,2,\ldots,n)$
such that for every $\beta$ with $p_\beta \neq 0$, either $\alpha_{\sigma(1)} > \beta_{\sigma(1)}$, or there exists an index $j, j \geq 2,$ satisfying
$\alpha_{\sigma(j)} > \beta_{\sigma(j)}$ and $\alpha_{\sigma(k)} = \beta_{\sigma(k)}$ for $1 \leq k \leq j-1.$ It is easy to see that every polynomial admits at least one admissible multi-index.

\begin{theorem}\cite{Dieu-2018}\label{Dieu} Let $p \in \R[x]$ be a polynomial of degree $d$ and let $\alpha \in \N^n$ be an admissible multi-index for $p$. There is a constant $C'$ depending only on $n$, such that for every $\delta>0$ and $r>0$ one has
\begin{equation} 
{\rm vol} (V_\delta(p) \cap \Delta_r) \leq C'[ \frac{4d}{|p_\alpha|^{1/|\alpha|}} \delta^{1/|\alpha|} r^{n-1} + (\frac{4d}{|p_\alpha|^{1/|\alpha|}} \delta^{1/|\alpha|})^n].
\end{equation}
\end{theorem}

More precise lower- and upper-bounds for $\delta \rightarrow 0$ are proved in \cite{Greenblatt-2010}. For our applications we isolate the following statement.

\begin{corollary}\label{vol} In the conditions of Theorem \ref{Dieu} there exists a positive constant $C$, depending only on $n$, such that
\begin{equation}\label{constant}
\frac{ {\rm vol} (V_\delta(p) \cap \Delta)}{\delta^{1/|\alpha|}} \leq C, 
\end{equation}
whenever $\delta < \frac{|p_\alpha|}{(4d)^{|\alpha|}}$.
\end{corollary}

A choice of an affine polynomial function shows that $C \geq 2$, in any dimension.

\section{Lewis estimates} Throughout this section $\Delta = [-1,1]^n$ denotes the cube in $\R^n$ and $d\lambda$ stands for the Lebesgue volume measure supported by $\Delta$. Consider a non-trivial polynomial $p(x) = \sum_{|\alpha| \leq d} p_\alpha x^\alpha$ and its sublevel set in the cube:
$$ E = \{ x \in \Delta, \  p(x) \geq 0 \}.$$
We remarked in the previous section that the characteristic function $\chi = \chi_E$ is extremal for the moment problem on $\Delta$ subject to the constraints
$g \in L^1(\lambda), 0 \leq g \leq 1, \lambda-a.e.$ .

Fix a function $f \in L^1(\lambda), f(x) > 0, \lambda-a.e.,$ and $\epsilon >0$. For a function $h$ denote $h_+(x) = \max (h(x),0).$ Define, following \cite{Lewis-1996}:
$$ \Lambda_f(\epsilon) = \inf \{ \int f \chi d\lambda,\ \ g \in L^1(\lambda), 0 \leq g \leq 1, \int g d\lambda \geq \epsilon\}.$$
The infimum is attained by the weak-$\ast$ compactness of the set of test functions in $L^\infty(\Delta)$.
A basic duality result in convex optimization states:
\begin{equation}\label{Fenchel}
\Lambda_f(\epsilon) \geq \sup_{s \geq 0}  ( s\epsilon - \int (s-f)_+ d \lambda).
\end{equation}
We apply this inequality to $f = |p|$ and small parameter $s$. Since the polynomial $p$ is not vanishing identically, its zero set has Lebesgue measure zero, hence
$|p| >0, \lambda-a.e.$. Fubini theorem yields:
$$ \int (s-|p|)_+ d\lambda = \int_{\{ 0 < |p(x)| \leq s\}} \int_0^{s-|p(x)|} dt d\lambda = $$ $$\int_0^s \lambda \{ x \in \Delta, \ |p(x)| \leq r \} dr.$$

Let $\alpha$ denote an admissible multi-index with respect to the polynomial $p$ and denote $\gamma = \frac{1}{|\alpha|}.$ In view of Corollary \ref{vol} one finds, for $s < \frac{|p_\alpha|}{(4d)^{|\alpha|}}$:
\begin{equation}
\int (s-|p|)_+ d\lambda \leq C \int_0^s r^\gamma dr = C_1 s^{\gamma+1},
\end{equation}
where the constant $C_1 = \frac{C}{\gamma+1}$ depends only on $p$, the multi-index $\alpha$ and the dimension $n$.

We return to inequality (\ref{Fenchel}) with the choice $s = t \epsilon^{|\alpha|}$. Note that 
$$ -|\alpha| -1 + (1+\gamma)|\alpha| = 0.$$ Hence
$$ \epsilon^{-|\alpha|-1} \Lambda_{|p|}(\epsilon) \geq t\epsilon^{|\alpha|+1 -|\alpha| - 1} - \epsilon^{-|\alpha|-1} \int ( t \epsilon^{|\alpha|} - |p|)_+ d\lambda \geq $$ $$
t - \epsilon^{-|\alpha| -1} C_1 t^{1+\gamma} \epsilon^{(1+\gamma)|\alpha|} = t - C_1 t^{1+\gamma}.
$$
For $t_0 = C^{-|\alpha|} = [C_1 (1+\gamma)]^{-1/\gamma}$ we find
$$  t_0 - C_1 t_0^{1+\gamma} = \frac{C^{-|\alpha|}}{1+|\alpha|} = C_2.$$
Therefore,
\begin{equation}
 \epsilon^{-|\alpha|-1} \Lambda_{|p|}(\epsilon) \geq C_2.
 \end{equation} 
 With this choice the validity threshold for the above inequality becomes:
 $$ \epsilon^{|\alpha|} < \frac{|p_\alpha|}{(4d)^{|\alpha|}} C^{|\alpha|}.$$
 As a final step we invoke an inequality valid for every $g \in L^1(\lambda), 0 \leq g \leq 1$:
 $$ \int (\chi - g)p d\lambda \geq \Lambda_{|p|}(|| \chi- g ||_{1,\lambda}).$$
 For the elementary proof we refer to Lemma 2.10 of \cite{Lewis-1996}. All in all we have proved the following refinement of the main result of Lewis \cite{Lewis-1996}.
 Recall that $s_\alpha(h) = \int x^\alpha h d\lambda$ denotes the moment of order $\alpha$ of a measurable function $h$.
 
 \begin{theorem}\label{main} Let $\Delta=[-1,1]^n$ denote the cube in $\R^n$ endowed with Lebesgue volume measure 
 and fix a degree $d \geq 1$.
 Let $p(X) = \sum_{|\beta| \leq d}p_\beta X^\beta$ be a non-constant polynomial and let $\alpha$ be an admissible multi-index with respect to $p$. Denote by
 $\chi$ the characteristic function of the sublevel set $p(x)\geq 0, x \in \Delta$.
 Then
 $$ \| \chi - g \|_1^{|\alpha|+1} \leq C^{|\alpha|} (1+|\alpha|) |\sum p_\beta (s_\beta(\chi) - s_\beta(g))|$$
 for every measurable function $g$ in the ball $\| \chi-g\|_1 \leq  \frac{|p_\alpha|^{1/|\alpha|}}{4d}C$, where the constant $C$ depends only on $n$.
 
 \end{theorem} 
 
 Having the moment space $\R^N$ (containing $s(g) = (s_\alpha(g))_{|\alpha|\leq d}$) endowed with a norm $\| \cdot \|$, the above inequality yields:
 $$  \| \chi - g \|_1 \leq K \| s(\chi) - s(g) \|^\frac{1}{|\alpha|+1},$$
 with a constant $K$ depending on the polynomial $p$, the admissible multi-index $\alpha$ and $n$.
 
 If the level set $p(x) = 0$ is smooth, then the co-area theorem implies along the same lines a Lipschitz estimate corresponding to $\alpha =0$
 in the above statement. See for details \cite{Lewis-1996}.

\section{Two dimensions} A closer look to the original $L$-problem of moments in one variable cannot avoid the Cauchy transform formulas and the rational approximation questions they raise. To fix ideas, consider the moment problem for measures of the form $g(t)dt$ with a Lebesgue measurable function
$g : [-1,1] \longrightarrow [0,1]$. The power moments $s_k(g) = \int_{-1}^1 g(t) t^n dt, \ \ k \geq 0,$ can be arranged into the generating analytic series
$$ \sum_{k=0}^\infty \frac{s_k(g)}{z^{k+1}} = - \int_{-1}^1 \frac{g(t) dt}{t-z}, \ \ |z|>1.$$
The right hand side is of course the analytic continuation of the series beyond its disk of convergence. The key property of these rather special generating series was discovered by Markov, via the formal exponential transform:
\begin{equation}
\exp [  -\sum_{k=0}^\infty \frac{s_k(g)}{z^{k+1}} ] = 1- \sum_{j=1}^\infty  \frac{t_k(g)}{z^{k+1}}.
\end{equation}
The new coefficients depend on the moments via a universal triangular system of polynomial equations $t_j(g) = R_j(s_0(g), \ldots, s_j(g)), \ j \geq 0.$
The positivity of the infinite Hankel matrix $[ t_{j+\ell}(g)]_{j,\ell=0}^\infty$ characterizes the moment sequence, in conjunction with additional shift conditions reflecting
the constraint on the support of $g$, see \cite{Ahiezer-Krein-1962}. 

In the above setting the measure $g(t)dt$ is determined by finitely many of its moments if and only if
there exists an integer $d$, such that $\det [ t_{j+\ell}(g)]_{j,\ell=0}^d = 0$; in which case we already know that $g$ is the sublevel set of a polynomial function, that is 
a finite collection of intervals. Moreover, in this case the exponential transform is a rational function
$$ \exp [  -\sum_{k=0}^\infty \frac{s_k(g)}{z^{k+1}} ] = \frac{Q(z)}{P(z)},$$
with $\deg Q + 1 = \deg P \leq d.$ To determine the polynomials $Q$ and $P$ one only needs the truncated transform $\exp [  \sum_{k=0}^d \frac{s_k(g)}{z^{k+1}} ]$
and the familiar Pad\'e approximation scheme. Ample details on the origins of this well charted chapter of constructive approximation theory can be found in the 
monograph \cite{Krein-Nudelman-1977}.

A two dimensional counterpart to the above classical setting has emerged as a byproduct of spectral analysis of a class of Hilbert space operators. Without entering into technical details,
systematized in the recent notes \cite{GP-2017}, we extract from there an outline of a shape reconstruction algorithm, see also \cite{GHMP-2000,P-NumerMath}. The frame is the unit disk
$\D$, with test space filled by measurable functions $g : \D \longrightarrow [0,1]$. We write the power moments in complex coordinates:
$$ s_{k\ell}(g) = \int_\D z^k \overline{z}^\ell g dA, \ \ k, \ell \geq 0,$$
where $dA$ stands for Lebesgue area measure on the disk $\D$. The choice of disk over square is made due to some simplifications resulting from working with complex variables.

The formal generating series and its exponential transform are
$$ \exp [ \frac{-1}{\pi} \sum_{k,\ell=0}^\infty \frac{s_{k\ell}(g)}{z^{k+1} \overline{z}^{\ell+1}} ] = 1-  \sum_{k,\ell=0}^\infty \frac{b_{k\ell}(g)}{z^{k+1} \overline{z}^{\ell+1}}.$$
We recognize above a double Cauchy transform
$$ E_g(z,\overline{z}) = \exp [  \frac{-1}{\pi} \int_\D \frac{g(\zeta) dA(\zeta)}{(\zeta-z)(\overline{\zeta}-\overline{z})}], \ \ |z| > 1.$$
The polarized exponential transform $E_g(z,\overline{w})$ can be extended via the above formula to a separately continuous function defined on $\C^2$, by adopting the convention
$\exp (-\infty) =0$ whenever the integral transform diverges (possibly at some diagonal points), see for details \cite{GP-2017}. Similar to the one variable case, the infinite matrix
$[b_{k\ell}(g)]_{k,\ell=0}^\infty$ is positive semi-definite. Moreover,
\begin{equation}\label{degenerate}
 \det [b_{k\ell}(g)]_{k,\ell=0}^d = 0
 \end{equation}
for some positive integer $d$, if and only if the original shade function $g$ is the characteristic function of a quadrature domain $\Omega$ contained in $\D$ \cite{P-1996}. By definition, a {\it quadrature domain}
is a bounded open set $\Omega \subset \C$ satisfying a Gaussian type quadrature
$$ \int_\Omega f(z) dA(z) = c_1 f(a_1) + \ldots+ c_d f(a_d), $$
valid for all complex analytic functions $f$ which are integrable on $\Omega$. Above the nodes $a_1,\ldots,a_d$ belong to $\Omega$ and the weights $c_1,\ldots,c_d$ are positive. Higher multiplicity nodes, that is derivatives of $f$, are permitted in such an identity.
For instance a disk is a quadrature domain, in view of Gauss mean value theorem. The conformal image of a disk by a rational function is also a quadrature domain.

Any quadrature domain is a principal semi-algebraic set, with an irreducible defining polynomial:
$ \Omega = \{ z \in \C, \ Q(z,\overline{z}) < 0\}$ (modulo a finite set), where
\begin{equation}\label{defining}
 Q(z,\overline{z}) = |P_d(z)|^2 - |P_{d-1}(z)|^2 - |P_{d-2}(z)|^2 - \ldots - |P_1(z)|^2 - |P_0(z)|^2,
 \end{equation}
 
with $P_j \in \C[z], 0 \leq j \leq d,$ and $\deg P_j = j, \ 0 \leq j \leq d.$ The degenerate situation (\ref{degenerate}) is reflected in the rationality of the exponential transform
\begin{equation}\label{rational}
 E_g(z,\overline{w}) = \frac{Q(z,\overline{w})}{P_d(z) \overline{P_d(w)}}, \ \ |z|, |w|  \rightarrow \infty,
 \end{equation}
and vice-versa, provided the degeneracy degree $d$ is chosen minimal. The nodes $a_1,\ldots,a_d$ of the mechanical quadrature are exactly the zeros of the leading polynomial $P_d(z)$. See for details \cite{P-1996}. We also note that quadrature domains are dense in Hausdorff metric among all bounded open subsets of the complex plane.

In general, for an open set $G \subset \D$, the exponential transform of its characteristic function $E_G = E_{\chi_G}$ shares the features of a numerically accessible, defining potential:
\begin{itemize}
\item $\lim_{z \rightarrow \infty} E_G(z,\overline{z}) = 1,$
\item $ E_G(z,\overline{z})$ is superharmonic and positive on $\C \setminus G$
\item $E_G(z,\overline{z}) \sim {\rm dist}(z, \partial G), \ \ z \rightarrow \partial G, \ z \notin G,$
\item $E_G(z,\overline{z})$ extends as a real analytic function acros analytic arcs of $\partial G$.
\end{itemize}
For instance, in the case of a disk $D(a,r)$ elementary computations yield:
$$E_{D(a,r)}(z,\overline{z}) = 1 - \frac{r^2}{|z-a|^2}, \ \ |z-a|>r.$$
As already mentioned, a similar finite determinateness (encoded in a rational expression) persists for all quadrature domains. For a quadrature domain $\Omega$ the above listed properties are satisfied by
the rational function (\ref{rational}).

These observations led to a shape reconstruction from moments algorithm, with structured rational approximations of the exponential transform as defining level sets close to the boundary \cite{GHMP-2000,P-NumerMath}. In what follows we add a few remarks derived form the general theorem proved in the present note, focused on the stability of this reconstruction scheme. Regardless to say that
the reconstruction is exact at finite degree for quadrature domains.

To start with we isolate a proposition of general interest. 

\begin{proposition}\label{diagonal}  Let $\epsilon >0$ and let $f, g: \D \longrightarrow [0,1]$ be two measurable functions satisfying $\| f - g \|_1 < \epsilon.$ 
Denote $K = ({\rm supp}(f) \cup {\rm supp}(g)$.
Then
\begin{equation}
|E_f(z,\overline{z}) - E_g(z,\overline{z})| \leq  \frac{2}{\pi {\rm dist}(z, K)^2} \| f - g \|_1,
\end{equation}
for $z \in \C \setminus K.$
\end{proposition} 

\begin{proof} Assume first that $f \leq g, a.e.$. In the range $z \in \C \setminus {\rm supp}(g)$ the representation
$$ E_g(z,\overline{z}) = \exp (\frac{-1}{\pi} \int \frac{ g(\zeta) dA(\zeta)}{|\zeta-z|^2}) $$ holds true. Consequently
$$ | E_f(z,\overline{z}) - E_g(z,\overline{z})| = E_f(z,\overline{z})|1-E_{g-f}(z,\overline{z})| \leq|1-E_{g-f}(z,\overline{z})|.$$
In its turn
$$ 1-E_{g-f}(z,\overline{z}) = \frac{1}{\pi} \int \frac{ (g-f)(\zeta) dA(\zeta)}{|\zeta-z|^2}) \int_0^1\exp (\frac{-1}{\pi} \int \frac{ t (g-f)(\zeta) dA(\zeta)}{|\zeta-z|^2}) dt.$$
For every $t$ the function $t(g-f)$ has values in the interval $[0,1]$, hence
$$ \int_0^1\exp (\frac{-1}{\pi} \int \frac{ t (g-f)(\zeta) dA(\zeta)}{|\zeta-z|^2}) dt \leq 1.$$
And
$$| \frac{-1}{\pi} \int \frac{ (g-f)(\zeta) dA(\zeta)}{|\zeta-z|^2})| \leq  \frac{1}{\pi {\rm dist}(z, K)^2} \| f - g \|_1.$$
In the general case one appeals to the function $\max(f,g)$ by applying the previous bounds to
$\max(f,g)-f$ and $\max(f,g)-g$.
\end{proof}

\begin{corollary}\label{mixed} In the conditions of the Proposition one has:
\begin{equation}
|E_f(z,\overline{w}) - E_g(z,\overline{w})| \leq C_3(R) \| f - g \|_1,
\end{equation}
for $|z|,|w| \geq R > 1$:
where $C_3(R) =  \frac{2}{\pi (R-1)^2}  \exp  (\frac{4}{\pi (R-1)^2} ).$
\end{corollary} 

\begin{proof}
In general, for a measurable function $h: \Delta \longrightarrow [0,1]$, the kernel $1-E_h(z,\overline{w})$ is positive semi-definite (c.f. \cite{P-1996}),
hence Cauchy-Schwartz inequality implies
$$ |1-E_h(z,\overline{w})|^2 \leq (1-E_h(z,\overline{z}))( 1-E_h(w,\overline{w})) \leq 1.$$
The rest of the proof is similar to the proof of the Proposition.
\end{proof} 

\begin{corollary}\label{b-gap} In the conditions of the Proposition, the displacement of the coefficients of the two exponential transforms is:
\begin{equation}
|b_{k\ell}(f) - b_{k\ell}(g)| \leq C_3(R)R^{k+\ell} \|f - g \|_1, \ k,\ell \geq 0,
\end{equation}
for any $R>1.$
\end{corollary}

As a final touch we exploit the superresolution estimates of last section, in the special case of a quadrature domain.
Let $\Omega \subset \D$ be a quadrature domain of order $d$, with characteristic function $\chi$ and defining polynomial 
equation $Q(z,\overline{z})<0$. The structure of $Q$ is described in (\ref{defining}).
We assume the leading polynomial $P_d$ is monic. The highest order term of $Q(z,\overline{z})$ is $|z|^{2d}$, hence the multi-index $(2d,0)$ is admissible for $Q$,
with coefficient equal to $1$. Theorem \ref{main} yields:

Let $g : \D \longrightarrow [0,1]$ be a measurable function. Then
$$ \| \chi-g\|_{1,\D} \leq e^{1/e} C | [\sum_{j=0}^d | \|P_j\|^2_{2,\Omega} -\|P_j\|^2_{2,gdA}| ]^{\frac{1}{2d+1}},$$
 provided
$\|\chi-g\|_1 \leq \frac{C}{4d}$, where $C, C \geq 2,$ is a universal constant  and $C_4 = e^{1/e} C.$ Notice that on the right hand side only complex moments of bidegree
less than or equal to $(d,d)$ appear. Proposition \ref{diagonal} and its corollaries provide effective bounds for the uniform gap between the two exponential transforms. In particular we infer the rational approximation evaluation:
$$ | E_g(z,\overline{w}) - \frac{Q(z,\overline{w})}{P_d(z)\overline{P_d(w)}} | \leq C_5(R) |\int_\D Q(u,\overline{u}) (\chi-g)(u) dA(u)|^{\frac{1}{2d+1}},$$
where $|z|,|w| \geq R > 1$ and $C_5(R) = C_4 C_3(R)$. We apply these results to the case of quadrature domains possessing the same nodes, a well studied scenario in fluid mechanics \cite{QD-2005}.

\begin{theorem}\label{two-domains} Let $\Omega_1, \Omega_2$ be quadrature domains contained in the unit disk, possessing the same nodes. Let $Q_j(z,\overline{z})$ be the defining polynomial of 
$\Omega_j, j =1,2,$ and let $P(z)$ be the monic polynomial of degree $d$ vanishing at the common nodes. There exists a constant $C$ independent of all data, such that
$$ \sup_{|z|<2} |Q_1(z,z) -Q_2(z,z)| \leq  4 C 3^{2d}   |\int_\D Q_1(u,\overline{u})(\chi_{\Omega_1} -\chi_{\Omega_2})(u) dA(u)|^{\frac{1}{2d+1}},$$
whenever $\|\chi_{\Omega_1} -\chi_{\Omega_2}\|_1 \leq \frac{C}{4d}$. 
\end{theorem}

Above $C$ is the same constant appearing in Corollary \ref{vol}.

\begin{proof} We take $R=2$ in Corollary \ref{mixed} and note that $C_4 C_3(R) \leq 4 C$. And we remark that a monic polynomial $P(z)$ of degree $d$, with all roots in 
the unit disk, satisfies $ \sup_{|z|\leq 2} |P(z)| \leq 3^d.$
\end{proof}

The non-uniqueness case of a continuous family of quadrature domains possessing the same nodes and weights is also
notable \cite{Gustafsson-1988}. In this case the power moments, and implicitely the rational exponential transform offer additional
parameter able to locate the quadrature domain. In the defining equation (\ref{defining}) the first {\it two} terms ($P_d(z)$ and
$P_{d-1}(z)$) determine the nodes and weights. Indeed, assuming
$$ \int_\Omega f(z) dA(z) = c_1 f(a_1) + \ldots+ c_d f(a_d), $$
for any analytic function integrable in $\Omega$, and expanding at infinity the identity:
$$ 1- \frac{P_{d-1}(z) \overline{P_{d-1}(w)}}{P_{d}(z) \overline{P_{d}(w)}} - \frac{P_{d-2}(z) \overline{P_{d-2}(w)}}{P_{d}(z) \overline{P_{d}(w)}} - \ldots = $$ $$1- \exp [  \frac{-1}{\pi} \int_\Omega \frac{ dA(\zeta)}{(\zeta-z)(\overline{\zeta}-\overline{w})}],$$
the coefficient of $1/(z\overline{w})$ is 
$$ - \gamma^2 = \frac{- {\rm Area}(\Omega)}{\pi} = - \frac{c_1+ c_2 + \ldots c_d}{\pi} ,$$
where $\gamma>0$ is the leading coefficient of $P_{d-1}$, while
the coefficient of $-1/{\overline{w}}$ is
$$  \gamma \frac{P_{d-1}(z)}{P_{d}(z)} =  \frac{1}{\pi} \int_\Omega \frac{ dA(\zeta)}{\zeta-z} = \sum_{j=1}^d \frac{c_j}{a_j -z}.$$

Therefore, if both domains $\Omega_1, \Omega_2$ appearing
in Theorem \ref{two-domains} share the same quadrature nodes and weights, then the left hand term \\
$|Q_1(z,z) -Q_2(z,z)|$
has bi-degree less than or equal to $(d-2,d-2)$.

The reconstruction algorithm proposed in \cite{GHMP-2000,P-NumerMath} starts with the complex moments $[s_{k\ell}(G)]_{k,\ell=0}^d$ of an unknown open set $G \subset \D$.
After performing the exponential transform of the truncated moment data one obtains the positive semi-definite matrix $X = [b_{k\ell}(G)]_{k,\ell=0}^d$. 

If $\det X = 0$ and $d$ 
is minimal with this property, there exists a null row vector 
$u = (p_0, p_1, \ldots, p_{d-1}, 1), u X = 0.$ Then $G$ is a quadrature domain of order $d$, with nodes at the roots of the polynomial
$P_d(z) = p_0 + p_1 z + \ldots+ p_{d-1} z^{d-1} + z^d$. The defining equation of $G$ is then the polynomial $Q(z,\overline{w})$ defined as the positive monomial part of the product:
$$ P_d(z) \overline{P_d(w)} [ (\frac{1}{z}, \frac{1}{z^2}, \ldots, \frac{1}{z^d}) X (\frac{1}{\overline{w}}, \frac{1}{\overline{w}^2}, \ldots, \frac{1}{\overline{w}^d})^T = Q(z,\overline{w}) + O(\frac{1}{z}, \frac{1}{\overline{w}}).$$

In case $\det X \neq 0$ one selects the eigenvector vector \\
$u = (p_0, p_1, \ldots, p_{d-1}, 1)$ of the matrix $X$ corresponding to the lowest eigenvalue. Then the same Pad\'e type scheme, with 
denominator $P_d(z)$ as before produces a defining equation $Q(z,\overline{z}) =0$ which approximates the boundary of $G$.

If the unknown open set $G$ is close in $(d,d)$ moments to a given quadrature domain $\Omega$, and the degree $d$ exceeds the order of $\Omega$, then Corollary \ref{b-gap} becomes instrumental
in evaluating the Hausdorff distance between $G$ and $\Omega$, similarly to the statement of the last Theorem. In this case however, the size of the first non-zero eigenvalue of the matrix
$[b_{k\ell}(\Omega)]_{k,\ell=0}^d$ will affect the gap estimate. It is worth recalling that every open set $G$ can be approximated in Hausdorff distance by a sequence of quadrature domains. 
We will elaborate the details in a separate work.

\end{document}